\pdfoutput=1
\RequirePackage{ifpdf}
\ifpdf 
\documentclass[pdftex]{sigma}
\else
\documentclass{sigma}
\fi

\numberwithin{equation}{section}

\newtheorem{Theorem}{Theorem}[section]
\newtheorem{Corollary}[Theorem]{Corollary}
 { \theoremstyle{definition}
\newtheorem{Remark}[Theorem]{Remark} }

\begin{document}

\allowdisplaybreaks

\newcommand{\arXivNumber}{1602.03693}

\renewcommand{\PaperNumber}{063}

\FirstPageHeading

\ShortArticleName{Symmetries of Lorentzian Three-Manifolds with Recurrent Curvature}

\ArticleName{Symmetries of Lorentzian Three-Manifolds\\ with Recurrent Curvature}

\Author{Giovanni CALVARUSO~$^\dag$ and Amirhesam ZAEIM~$^\ddag$}

\AuthorNameForHeading{G.~Calvaruso and A.~Zaeim}

\Address{$^\dag$~Dipartimento di Matematica e Fisica ``E. De Giorgi'', Universit\`a del Salento,\\
\hphantom{$^\dag$}~Prov. Lecce-Arnesano, 73100 Lecce, Italy}
\EmailD{\href{mailto:giovanni.calvaruso@unisalento.it}{giovanni.calvaruso@unisalento.it}}

\Address{$^\ddag$~Department of Mathematics, Payame Noor University, P.O. Box 19395-3697, Tehran, Iran}
\EmailD{\href{mailto:zaeim@pnu.ac.ir}{zaeim@pnu.ac.ir}}

\ArticleDates{Received February 12, 2016, in f\/inal form June 17, 2016; Published online June 26, 2016}

\Abstract{Locally homogeneous Lorentzian three-manifolds with recurrect curvature are special examples of Walker manifolds, that is, they admit a parallel null vector f\/ield. We obtain a full classif\/ication of the symmetries of these spaces, with particular regard to symmetries related to their curvature: Ricci and matter collineations, curvature and Weyl col\-lineations. Several results are given for the broader class of three-dimensional Walker manifolds.}


\Classification{53C50; 53B30}

\section{Introduction}

A {\em Walker manifold} is a pseudo-Riemannian manifold $(M,g)$ admitting a degenerate parallel distribution. Such a phenomenon is peculiar to the case of indef\/inite metrics. As such, it is responsible for many special geometric properties of pseudo-Riemannian manifolds which do not have any Riemannian counterpart, and has been investigated by several authors under dif\/ferent points of view. The monograph~\cite{bookW} is a well-written recent survey on Walker manifolds and the various related research areas.

Lorentzian three-manifolds admitting a parallel degenerate line f\/ield have been studied in~\cite{ChGRVA}. These Lorentzian metrics are described in terms of a suitable system of local coordinates $(t,x,y)$ and form a large class, depending on an arbitrary function $f(t,x,y)$. The case of {\em strictly Walker manifolds}, where the parallel degenerate line f\/ield is spanned by a parallel null vector f\/ield, is characterized by condition $f=f(x,y)$. The results of~\cite{ChGRVA} have been recently used in~\cite{GGN} to obtain a complete classif\/ication of the models of locally homogeneous Lorentzian three-manifolds with recurrent curvature.

The aim of this paper is to investigate symmetries of these Lorentzian spaces. If $(M,g)$ denotes a Lorentzian manifold and $T$ a tensor on $(M,g)$, codifying some either mathematical or physical quantity, a {\em symmetry} of~$T$ is a one-parameter group of dif\/feomorphisms of $(M,g)$, leaving~$T$ invariant. As such, it corresponds to a vector f\/ield $X$ satisfying $\mathcal L_X T=0$, where~$\mathcal L$ denotes the Lie derivative. Isometries are a well known example of symmetries, for which $T=g$ is the metric tensor. The corresponding vector f\/ield~$X$ is then a Killing vector f\/ield. Homotheties and conformal motions on $(M,g)$ are again examples of symmetries. In recent years, symmetries related to the curvature of the manifold have been investigated. Among them: {\em curvature collineations} (where $T$=$R$ is the curvature tensor), {\em Weyl collineations} ($T$=$W$ being the Weyl conformal curvature tensor) and {\em Ricci collineations}, for which $T$=$\varrho$ is the Ricci tensor. We may refer to the monograph~\cite{Hall} for further information and references on symmetries. Ricci {and curvature} collineations have been investigated in several classes of Lorentzian manifolds (see, for example, \cite{Heic,CZ1,CZ2,CSVV,CHK,FPP,Ha,HC,HLP,HRV,KR,L,TA} and references therein). Because of their physical relevance, in most cases curvature symmetries have been studied for some space-times. Moreover, the three-dimensional case has also been considered as an interesting source of examples and dif\/ferent behaviours (see, for example,~\cite{CSVV}).

A {\em matter collineation} of a Lorentzian manifold $(M,g)$ is a vector f\/ield $X$, corresponding to a symmetry of the energy-momentum tensor $T = \varrho- \frac{1}{2}\tau g$, where $\tau$ denotes the scalar curvature. Matter collineations are more relevant from a physical point of view~\cite{CS,CCV}, while Ricci collineations have a more clear geometrical signif\/icance, since $\varrho$ is naturally deduced from the connection of the metric~\cite{HRV}. These physical and geometrical meanings do coincide in a~special case, namely, for metrics with vanishing scalar curvature. And this is exactly the case for any strictly Walker three-manifold~\cite{ChGRVA}.

We shall obtain complete classif\/ications of curvature and Ricci ($\equiv$ matter) collineations of homogeneous Lorentzian three-manifolds with recurrent curvature. In Section~\ref{section2} we shall give some basic information about Walker three-manifolds and curvature symmetries. In Section~\ref{section3} we then investigate symmetries of an arbitrary strictly Walker three-manifold. Since the function $f=f(x,y)$ determining the metric tensor here is arbitrary, one cannot expect to obtain these symmetries explicitly in the general case. However, we describe the sets of partial dif\/ferential equations describing the dif\/ferent symmetries and use them to give some explicit examples of proper symmetries. Then, in Section~\ref{section4} we shall completely classify the symmetries of homogeneous Lorentzian three-manifolds with recurrent curvature.
All calculations have also been checked using {\em Maple16$^\copyright$}.

\section{Preliminaries}\label{section2}

\subsection{Three-dimensional Walker metrics}

We shall essentially follow the notations used in \cite{ChGRVA}. A three-dimensional Lorentzian mani\-fold~$M$ admitting a parallel degenerate line f\/ield has local coordinates $(t,x,y)$, such that with respect to the local frame f\/ield
$\{ \partial_t ,\partial_x,\partial_y \}$ the Lorentzian metric is given by
\begin{gather*}
g_f =\left(
\begin{matrix}
0 & 0 & 1 \\
0 & \varepsilon & 0 \\
1 & 0 & f(t,x,y)
\end{matrix}
\right),
\end{gather*}
for some function $f(t,x,y)$. In the above expression, $\varepsilon =\pm 1$. However, it is easily seen that by reversing the metric and changing the sign of the coordinate~$x$, without loss of generality one can reduce to the case $\varepsilon=1$ (as it was done, for example, in~\cite{GGN}).

The parallel degenerate line f\/ield is spanned by $\partial_{t}$, and the existence of a parallel null vector $U = \partial_{t}$ ({\em strictly Walker metric}) is {characterized} by the independence of {the function $f$} of the variable~$t$~\cite{W}. Therefore, with respect to local coordinates~$(t,x,y)$, the general form of a strictly Walker metric is given by
\begin{gather}\label{metric}
g_f =\left(
\begin{matrix}
0 & 0 & 1 \\
0 & 1 & 0 \\
1 & 0 & f(x,y)
\end{matrix}
\right),
\end{gather}
for an arbitrary smooth function $f$. With respect to the coordinate basis $\{\partial_t,\partial_x,\partial_y\}$, the Levi-Civita connection $\nabla$ and curvature $R$ of the metric $g_f$ described by~\eqref{metric} are completely determined by the following possibly non-vanishing components (see also~\cite{ChGRVA}):
\begin{gather}\label{connection}
\nabla_{\partial_x} \partial_y = \tfrac{1}{2} f_x \partial_t, \qquad \nabla_{\partial_y} \partial_y = \tfrac{1}{2} f_y \partial_t - \tfrac{1}{2} f_x \partial_x
\end{gather}
and
\begin{gather}\label{curvature}
R(\partial_x,\partial_y)\partial_x = \tfrac{1}{2} f_{xx} \partial_t, \qquad R(\partial_x,\partial_y)\partial_y = -\tfrac{1}{2} f_{xx} \partial_x,
\end{gather}
where $R(X,Y)=[\nabla_X,\nabla_Y]-\nabla_{[X,Y]}$. From \eqref{connection} and \eqref{curvature}, a straightforward calculation yields that the covariant derivative of the curvature tensor is completely determined by the possibly non-vanishing components
\begin{alignat}{3}
& (\nabla _{\partial_x} R) (\partial_x,\partial_y)\partial_x=\tfrac{1}{2} f_{xxx} \partial_t,\qquad&& (\nabla _{\partial_x} R) (\partial_x,\partial_y)\partial_y=-\tfrac{1}{2} f_{xxx} \partial_x, & \nonumber\\
& (\nabla_{\partial_y} R) (\partial_x,\partial_y)\partial_x=\tfrac{1}{2} f_{xxy} \partial_t ,\qquad && (\nabla_{\partial_y} R) (\partial_x,\partial_y)\partial_y=-\tfrac{1}{2} f_{xxy} \partial_x .& \label{dcurvature}
\end{alignat}
Either by \eqref{dcurvature} or by direct calculation dif\/ferentiating the Ricci identity, it is easily seen that three-dimensional (strictly) Walker metrics have {\em recurrent curvature}, that is, in a neighborhood of any point of non-vanishing curvature, one has $\nabla R =\omega \otimes R$, for a suitable one-form $\omega$. Since we are interested in the study of the nonf\/lat examples with recurrent curvature, throughout the paper we shall assume that $f_{xx} \neq 0$ at any point.

In local coordinates $(t,x,y)$, the Ricci tensor $\varrho$ of any metric \eqref{metric} is given by
\begin{gather}\label{Ricci}
\varrho = \left( \begin{matrix}
0 & 0 & 0 \\
0 & 0 & 0 \\
0 & 0 & -\frac{1}{2} f_{xx}
\end{matrix} \right).
\end{gather}
A pseudo-Riemannian manifold $(M,g)$ is said to be {\em locally homogeneous} if for any pair of points $p,q \in M$ there exist a neighbourhood~$U$ of~$p$, a neighbourhood $V$ of $q$ and an isometry $\phi\colon U\to V$. Hence, locally homogeneous manifolds \lq\lq look the same\rq\rq \ around each point. For any given class of pseudo-Riemannian manifolds, it is a natural problem to determine its locally homogeneous examples. {Locally homogeneous} examples among three-dimensional Walker metrics have been investigated in~\cite{GGN} (see also~\cite{BCL}). Rewriting the classif\/ication obtained in \cite{GGN} in terms of coordinates $(t,x,y)$ used in~\eqref{metric}, we have the following.

\begin{Theorem}[\cite{GGN}]\label{homex}
Locally homogeneous Lorentzian three-manifolds of recurrent curvature naturally divide into three classes. They correspond to one of the following types of $($strictly$)$ Walker metrics, as described in~\eqref{metric}:
\begin{itemize}\itemsep=0pt
\item[$I)$] $\mathcal N _b$, defined by taking $f(x,y)=-2b^{-2}e^{bx}$, for some real constant $b \neq 0$;
\item[$II)$] $\mathcal P _c$, defined by taking $f(x,y)=-x^{2}\alpha (y)$, where $\alpha >0$ satisfies {$\alpha ' _y=c \alpha^{3/2}$} for some real constant $c$;
\item[$III)$] $\mathcal{CW} _{\varepsilon}$, defined by taking $f(x,y)=-\varepsilon x^2$, where $\varepsilon =\pm 1$.
\end{itemize}
\end{Theorem}

\subsection{Curvature and Ricci collineations}

Let $(M,g)$ denote a pseudo-Riemannian manifold (in particular, a Lorentzian one). A vector f\/ield~$X$ on~$M$ preserving its metric tensor $g$, the corresponding Levi-Civita connection $\nabla$, its curvature tensor $R$ or its Ricci tensor $\varrho$, is respectively known as a {\em Killing vector field}, an {\em affine vector field}, a~{\em curvature collineation} or a~{\em Ricci collineation}.

It is obvious that if $X$ preserves $g$ (respectively,~$\nabla$,~$R$), then it also preserves $\nabla$ (respective\-ly,~$R$,~$\varrho$), but the converse does not hold in general. Homothetic vector f\/ields (i.e., vector f\/ields~$X$ satisfying $\mathcal L _X g=\lambda g$ for some real constant $\lambda$) are again necessarily curvature collineations (in particular, Ricci collineations). For this reason, we are specif\/ically interested in the existence of {\em proper} Ricci and curvature collineations, namely, the ones which are not homothetic (and hence, not Killing). Thus, we also need to specify which are the Killing, af\/f\/ine and homothetic vector f\/ields, which is an interesting problem on its own, due to the natural geometric meaning of such symmetries.

Conditions def\/ining Ricci and curvature collineations are formally similar to the ones def\/ining Killing or af\/f\/ine vector f\/ields. However, they may show some deeply dif\/ferent behaviours. In fact (see, for example, \cite{Hall,HRV}):
\begin{itemize}\itemsep=0pt
\item[(a)] Killing and af\/f\/ine vector f\/ields are smooth (provided they are at least $C^1$). However, for any positive integer $k$, there exist Lorentzian metrics admitting Ricci {(and curvature)} collineations, which are $C^k$ but not $C^{k+1}$.
\item[(b)] Unlike Killing and af\/f\/ine vector f\/ields, Ricci (and curvature) collineations form a vector space which may be inf\/inite-dimensional and (because of the above point~(a)) is not necessarily a Lie algebra. In fact, if $X$, $Y$ are Ricci (curvature) collineations, then $[X,Y]$ might not be dif\/ferentiable.
\item[(c)] While Killing and af\/f\/ine vector f\/ields agreeing in the neighbourhood of a point must coincide everywhere, two Ricci (respectively, curvature) collineations that agree on an non-empty subset of~$M$ may not agree on~$M$, since they are not uniquely determined by the value of~$X$ and its covariant derivatives of any order at a point.
\end{itemize}

Observe that the above item~(b), as concerns the possibility of the vector space of Ricci collineations to be inf\/inite-dimensional, refers to cases where the Ricci tensor $\varrho$ is necessarily degenerate (as it is always the case, for example, for three-dimensional strictly Walker metrics). On the other hand, if $\varrho$ (respectively, $T=\varrho-\frac{1}{2}\tau g$) is nondegenerate, then Ricci (respectively, matter) collineations form a f\/inite-dimensional Lie algebra of smooth vectors. In fact, in such a~case, they are exactly the Killing vector f\/ields of the nondegenerate metric tensor $\varrho$.

\section{Symmetries of Walker three-manifolds}\label{section3}

Observe that any three-dimensional strictly Walker metric is already equipped in a natural way with the strongest possible symmetry. In fact, it admits a {\em parallel} vector f\/ield $U=\partial_t$. We shall now investigate the dif\/ferent kinds of symmetries of these metrics, starting with Killing, homothetic and af\/f\/ine vector f\/ields. The descriptions we obtain are given in the following.

\begin{Theorem}\label{nablasymm}
Let $(M,g_f)$ be a three-dimensional strictly Walker manifold, where $g_f$ is described in the local coordinates $(t,x,y)$ by equation~\eqref{metric}. A~smooth vector field $X=X_1\partial_t+X_2\partial_x+X_3\partial_y$ is
\begin{itemize}\itemsep=0pt
\item[$i)$] a Killing vector field if and only if
\begin{gather}\label{eqKillvector}
X_1=-c_1t-{xf_1 '(y)}+f_2(y),\qquad X_2=f_1(y),\qquad X_3=c_1y+c_2,
\end{gather}
where $f_1$, $f_2$ are smooth functions on $M$, satisfying
\begin{gather}\label{eqKill}
2c_1f-2f_1''(y)x+2f_2 '(y)+f_1(y)\partial_xf+(c_1y+c_2)\partial_yf.
\end{gather}

\item[$ii)$] a homothetic, non-Killing vector field if and only if
\begin{gather}\label{eqHomvector}
X_1=\eta t-c_1t- {xf_1 '(y)}+f_2(y),\qquad X_2=\frac\eta2 x+f_1(y),\qquad X_3=c_1y+c_2,
\end{gather}
where $\eta \neq 0$ is a real constant and
\begin{gather}\label{eqHom}
(2c_1-\eta)f-2f_1''(y)x+2f_2 '(y)+\left(\frac\eta2x+f_1(y)\right)\partial_xf+(c_1y+c_2)\partial_yf=0.
\end{gather}

\item[$iii)$] an affine Killing vector field if and only if
\begin{gather}\label{eqAffvector}
 X_1=c_3t-xf_1 '(y)+f_2(y),\qquad X_2=\frac{c_1+c_3}{2}x+f_1(y),\qquad X_3=c_1y+c_2,
\end{gather}
where
\begin{gather}
(c_1-c_3)f-2f_1 ''(y)x+2f_2 '(y)+\left(\frac{c_1+c_3}{2}x+f_1(y)\right)\partial_xf\nonumber\\
\qquad{} +(c_1y+c_2)\partial_yf+c_4=0.\label{eqAff}
\end{gather}
\end{itemize}
\end{Theorem}

\begin{proof}
We start from an arbitrary smooth vector f\/ield $X=X_1\partial_t+X_2\partial_x+X_3\partial_y$ on the three-dimensional strict Walker manifold $(M,g_f)$, where $g_f$ is described by equation~\eqref{metric}, and calculate $\mathcal L_Xg_f$. Then, $X$ satisf\/ies $\mathcal L_Xg_f=\eta g_f$ for some real constant $\eta$ if and only if the following system of partial dif\/ferential equations is satisf\/ied:
\begin{gather}
\partial_tX_3=0,\qquad \partial_xX_2=\frac\eta2,\qquad \partial_xX_3+\partial_tX_2=0,\qquad
\partial_yX_3+\partial_t X_3f+\partial_tX_1=\eta,\nonumber\\
 f\partial_xX_3+\partial_xX_1+\partial_yX_2=0,\qquad 2\partial_yX_3f+2\partial_yX_1+X_2\partial_xf+X_3\partial_yf=\eta f.\label{sys1}
\end{gather}
We then proceed to integrate \eqref{sys1}. From the f\/irst three equations in \eqref{sys1} we get $X_2=\frac\eta2x-a_1(y)t+f_1(y)$ and $X_3=a_1(y)x+b_1(y)$. Then, the fourth equation in \eqref{sys1} yields $X_1=\eta t-a_1 '(y) tx-b_1 '(y)t+f_4(x,y)$. Substituting this into the f\/ifth equation, we get
\begin{gather*}2a_1 '(y)t=a_1(y)f+\partial_xf_4(x,y)+f_2 '(y),\end{gather*}
which must hold for all values of $t${, implying} that $a_1(y)=c_1$ is a constant. Now, the last equation in~\eqref{sys1} gives
\begin{gather*}
(c_1\partial_xf+2b_1 ''(y))t =(2b_1 '(y)-\eta f)f+2\partial_yf_4(x,y)+\left(\frac\eta2 x+f_2(y)\right)\partial_xf \\
\hphantom{(c_1\partial_xf+2b_1 ''(y))t =}{} +(c_1x+f_3(y))\partial_yf,
\end{gather*}
which immediately yields that $c_1\partial_xf+2b_1 ''(y)=0$ and so, $c_1\partial_{xx}^2f=0$. Since we assumed $\partial_{xx}^2f\neq 0$, we then have $c_1=0$ and integrating $b_1 ''(y)=0$ we get $b_1(y)=c_2y+c_3$. On the other hand, from the f\/ifth equation in~\eqref{sys1} we now have $f_4(x,y)=-f_2 '(y)x+f_5(y)$ and the last equation gives
\begin{gather*}
(2c_2-\eta)f-2f_2 ''(y)x+2f_5 '(y)+\left(\frac\eta2 x+f_2(y)\right)\partial_xf+(c_2y+c_3)\partial_yf=0.
\end{gather*}
This proves the statement~i) in the case $\eta=0$ and the statement~ii) if we assume $\eta \neq 0$.

With regard to af\/f\/ine vector f\/ields, expressing condition $\mathcal L_X\nabla=0$ in the coordinate basis $\{\partial_t,\partial_x,\partial_y\}$, we get the following system of partial dif\/ferential equations:
\begin{gather}
\partial_{tt}^2X_1=\partial_{tt}^2X_2=\partial_{xx}^2X_2=\partial_{tx}^2X_2=\partial_{tt}^2X_3=\partial_{xx}^2X_3=\partial_{tx}^2X_3=\partial_{ty}^2X_3=0,
\nonumber\\
\partial_{xx}^2X_1+\partial_xX_3\partial_xf=0,\nonumber\\
2\partial_{tx}^2X_1+\partial_tX_3\partial_xf=0,\nonumber\\
2\partial_{ty}^2X_2-\partial_tX_3\partial_xf=0,\nonumber\\
2\partial_{xy}^2X_3-\partial_tX_3\partial_xf=0,\nonumber\\
2\partial_{ty}^2X_1+\partial_tX_3\partial_yf+\partial_tX_2\partial_xf=0,\nonumber\\
2\partial_{xy}^2X_2-\partial_xX_3\partial_xf-\partial_tX_2\partial_xf=0,\nonumber\\
2\partial_{yy}^2X_3+\partial_xX_3\partial_xf-\partial_tX_3\partial_yf=0,\nonumber\\
2\partial_yX_3\partial_xf-2\partial_{yy}^2X_2-\partial_xX_2\partial_xf+\partial_tX_2\partial_yf+X_2\partial_{xx}^2f+X_3\partial_{xy}^2f=0,\nonumber\\
 \partial_yX_3\partial_xf+2\partial_{xy}^2X_1+\partial_xX_3\partial_yf+\partial_xX_2\partial_xf-\partial_tX_1\partial_xf
 +X_2\partial_{xx}^2f+X_3\partial_{xy}^2f=0,\nonumber\\
 2\partial_yX_3\partial_yf+2\partial_{yy}^2X_1+\partial_xX_1\partial_xf+2\partial_yX_2\partial_xf-\partial_tX_1\partial_yf
 +X_2\partial_{xy}^2f+X_3\partial_{yy}^2f=0.\label{sys2}
\end{gather}
As for the above system \eqref{sys1}, we then proceed to integrate~\eqref{sys2}. From the f\/irst equation we get $X_3=c_1t+a_1(y)x+f_2(y)$ and then the f\/ifth equation yields $2a_1 '(y)=c_1\partial_xf$, so that $c_1\partial_{xx}^2f=0$ and so, $c_1=0$. Then, $a_1(y)=c_2$ is a constant.

Integrating the third and fourth equations (taking into account the f\/irst one) we get $X_1=f_3(y)t+f_4(x,y)$, $X_2=c_3t+f_5(y)x+f_6(y)$. The sixth equation then gives $f_3 '(y)+c_3\partial_xf=0$, which, by the same argument above, yields $c_3=0$ and $f_3(y)=c_4$.

Similarly, the eighth equation becomes $2f_2 ''(y)+c_2\partial_xf=0$, which gives $c_2=0$ and $f_2(y)=c_5y+c_6$. The second equation now reads $\partial_{xx}^2f_4(x,y)=0$, and the seventh leads to $f_5 '(y)=0$. So, $f_4(x,y)=f_7(y)x+f_8(y)$ and $f_5(y)=c_7$.

By the ninth and tenth equations we then have $f_6 ''(y)+f_7 '(y)+\frac12(2c_7-c_4-c_5)\partial_xf=0$, so that $c_7=\frac{c_4+c_5}{2}$ and $f_7(y)=-f_6 '(y)+c_8$. Integrating the tenth equation with respect to the variable~$x$, we get
\begin{gather*}
(c_5-c_4)f-2f_6 ''(y)x+\left(\frac{c_4+c_5}{2}x+f_6(y)\right)\partial_xf+(c_5y+c_6)\partial_yf+f_9(y)=0.
\end{gather*}
We dif\/ferentiate the above equation with respect to $y$ and subtract the eleventh equation, obtaining $f_9 '(y)-2f_8 ''(y)=c_8\partial_xf$, which immediately leads to $c_8=0$ and $f_9(y)=2f_8 '(y)+c_9$. The statement follows after we {suitably rename} the remaining constants and functions.
\end{proof}

\begin{Remark}Since $f=f(x,y)$ is an arbitrary smooth function of two variables, we cannot integrate equations~\eqref{eqKill},~\eqref{eqHom} and \eqref{eqAff} of Theorem~\ref{nablasymm} in full generality. However, it is well known that the Lie algebras of Killing, homothetic and af\/f\/ine vector f\/ields are f\/inite-dimensional.

Therefore, on the one hand, for any prescribed function $f(x,y)$ these equations force the sets of Killing, homothetic and af\/f\/ine vector f\/ields of $(M,g_f)$ to depend on a f\/inite number of real parameters. On the other hand, they allow us to determine special functions $f$, for which we can f\/ind some explicit examples of homothetic non-Killing and af\/f\/ine non-homothetic vector f\/ields.

At the end of this {section} we shall illustrate these results calculating the symmetries of an arbitrary locally conformally f\/lat strictly Walker metric. In the next section we shall consider the functions~$f$ determining the locally homogeneous examples of three-dimensional strictly Walker manifolds. Further explicit examples can be determined by direct calculation.
\end{Remark}

\begin{Remark}[homothetic f\/ixed points]
The existence on a Lorentzian manifold $(M,g)$ of {\em homothetic fixed points}, that is, of a non-trivial homothetic vector f\/ield $X$ which vanishes at a~point $m\in M$, has some important consequences on the structure of the manifold itself. Dif\/ferent conclusions can be deduced depending on whether~$m$ is an isolated f\/ixed point or not. In the latter case, the zeroes of $X$ form a null geodesic, and the resulting metric is a kind of plane wave, whose conformal vector f\/ields can be determined. Interesting studies of the link between homothetic and conformal vector f\/ields (and their f\/ixed points) and the geometry the metrics can be found in~\cite{Ale,Beem,HC,HLP}. The above Theorem~\ref{nablasymm} and the special cases described in Theorem~\ref{cflatW} and in Section~\ref{section4}, allow us to discuss the existence of homothetic f\/ixed points for all three-dimensional Walker metrics, and gives a unif\/ied treatment for a large class of three-dimensional manifolds, where all dif\/ferent behaviours can occur, from metrics with no proper homothetic vector f\/ields, to cases where homothetic f\/ixed points occur and can be explicitly determined.
\end{Remark}

We now turn our attention to curvature collineations and prove the following.

\begin{Theorem}\label{RicciCol}
Let $X=X_1\partial_t+X_2\partial_x+X_3\partial_y$ be an arbitrary smooth vector field on the strictly Walker manifold $(M,g_f)$, where~$g_f$ is described as in \eqref{metric}. Then:
\begin{itemize}\itemsep=0pt
\item[$i)$] $X$ is a Ricci collineation if and only if one of the following cases occurs:
\begin{itemize}\itemsep=0pt
\item[$(a)$] $f$ is arbitrary and
\begin{gather*}X_2=-\frac{2f_1 '(y)\partial_{xx}^2f+f_1(y)\partial_{xxy}^3f}{\partial_{xxx}^3f}, \qquad X_3=f_1(y),\end{gather*}
where $f_1$ is an arbitrary smooth function on $M$, and the Ricci collineation is defined in the open subset where $\partial_{xxx}^3f \neq 0$.
\item[$(b)$] $f=f_1(y)x^2+f_2(y)x+f_3(y)$ and \begin{gather*}X_3=\frac{c_1}{\sqrt{|f_1(y)|}}.\end{gather*}
\end{itemize}

\item[$ii)$] $X$ is a curvature collineation if and only if~$X$ is a special Ricci collineation of one of the following types:
\begin{itemize}\itemsep=0pt
\item[$(a)^\prime$] type $(a)$ with
\begin{itemize}\itemsep=0pt
\item[$\bullet$] either $X_1=X_1 (y)$, $X_2=X_3=0$, or
\item[$\bullet$] $f(x,y)=f_2(x)f_3(y)+f_4(y)x+f_5(y)$, $f_1(y)=\frac{c_1}{\sqrt{|f_3(y)|}}$ and
\begin{gather*} X_1=\frac{c_1f_3 '(y)}2f_3(y)\sqrt{|f_3(y)|}t+f_6(y).\end{gather*}
\end{itemize}

\item[$(b)^\prime$] type $(b)$ with $X_1=2f_4(y)t+\frac{c_1f_1 '(y)}2f_1(y)\sqrt{|f_1(y)|}t-\frac12f_4 '(y) x^2-f_5 '(y) x+f_6(y)$ and $X_2=f_4(y)x+f_5(y)$.
\end{itemize}
\end{itemize}
\end{Theorem}

\begin{proof}
Because of equation~\eqref{Ricci}, a smooth vector f\/ield $X=X_1\partial_t+X_2\partial_x+X_3\partial_y$ on a strictly Walker manifold $(M,g_f)$ is a Ricci collineation if and only if
\begin{gather}\label{sys4}
 \partial_{xx}^2f \partial_tX_3=0,\qquad \partial_{xx}^2f\partial_xX_3=0,\qquad
 2\partial_{xx}^2f\partial_yX_3+ \partial_{xxx}^3f X_2+ \partial_{xxy}^3f X_3=0.
\end{gather}
As we already mentioned, we are always assuming that $\partial_{xx}^2f\neq 0$. Consequently, from the f\/irst two equations in~\eqref{sys4} we have $X_3=X_3(y)$, and the third equation becomes
\begin{gather}\label{mainric}
2\partial_{xx}^2f X_3 '(y)+\partial_{xxx}^3f X_2+\partial_{xxy}^3f X_3(y)=0.
\end{gather}
In the open subset where $\partial_{xxx}^3f \neq 0$, from the above equation~\eqref{mainric} we get at once the case~(a). Case~(b) is obtained as a~special solution of \eqref{mainric}, assuming that $\partial_{xxx}^3f =0$.

We then consider curvature collineations, starting from an arbitrary Ricci collineation as described in cases~(a) and~(b) and requiring the additional condition $\mathcal{L}_X R=0$. Calculations are of the same kind for all these cases. For this reason, we report the details only for case~(b).

So, consider a Ricci collineation $X=X_1\partial_t+X_2\partial_x+\frac{c_1}{\sqrt{|f_1(y)}|}\partial_y$, where $f=f_1(y)x^2+f_2(y)x+f_3(y)$. In particular, calculating the condition $\mathcal{L}_X R=0$ on the pairs of coordinate vector f\/ields $\partial_t$, $\partial_x$, $\partial_y$, we f\/ind that $X$ is a curvature collineation if and only if the following equations hold:
\begin{gather*}
\partial_yX_2+\partial_xX_1=\partial_tX_2=0,\qquad \frac{c_1f_1 '(y)}{2\sqrt{|f_1(y)|}}+f_1(y)\left(2\partial_xX_2-\partial_tX_1\right)=0.
\end{gather*}
It easily follows from the f\/irst of the above equations that $\partial_xX_1$ and $X_2$ are functions of the variables $(x,y)$. Since $f_1(y)\neq 0$, dif\/ferentiating with respect to $x$ the second of the above equations we get $\partial_{xx}^2X_2=0$ and so, $X_2=f_4(y)x+f_5(y)$. Now, again the second equation gives $X_1=\frac{c_1f_1 '(y)}{2f_1(y)\sqrt{|f_1(y)}|}t+2f_4(y)t+f_6(x,y)$. Then, since $\partial_xX_1=-\partial_yX_2$, we conclude that $f_6(x,y)=-\frac12f_4 '(y) x^2-f_5 '(y) x+f_7(y)$ and this ends the proof.
\end{proof}

Observe that taking $X_2=X_3=0$, all equations in~\eqref{sys4} are satisf\/ied. Therefore, $X=X_1 \partial_{t}$ is a Ricci collineation for any arbitrary smooth function $X_1=X_1(t,x,y)$, and (by case~(a)$^\prime$) a~curvature collineation for any smooth function $X_1=X_1(y)$. This implies at once the following.

\begin{Corollary}
For any strictly Walker three-manifold $(M,g_f)$, the Lie algebras of smooth Ricci collineations and smooth curvature collineations are infinite-dimensional. In particular, each of these spaces admits proper Ricci and curvature collineations.
\end{Corollary}

We end this section calculating the symmetries of a locally conformally f\/lat strictly Walker three-manifold. By direct calculations of the Cotton tensor of a strictly Walker three-mani\-fold~$(M,g_f)$ (see also~\cite{ChGRVA}), it is easily seen that this manifold is locally conformally f\/lat if and only if $\partial_{xxx}^3f$ vanishes identically, that is, when the def\/ining function is of the form $f(x,y)=p(y)x^2+q(y)x+r(y)$ (with $p(y) \neq 0$ in order to avoid the f\/lat case). We now prove the following.

\begin{Theorem}\label{cflatW}
Let $X=X_1\partial_t+X_2\partial_x+X_3\partial_y$ be an arbitrary smooth vector field on a conformally flat strictly Walker manifold $(M,g_f)$, where $g_f$ is described as in \eqref{metric} with $f(x,y)=p(y)x^2+q(y)x+r(y)$ $(p(y)\neq 0)$. Then, $X$ is:
\begin{itemize}\itemsep=0pt
\item[$i)$] a Killing vector field if and only if $(c_1y+c_2)p(y)'+2c_1p(y)=0$ and~$X$ is described as in~\eqref{eqKillvector}, with $f_1$, $f_2$ explicitly determined as solutions of
\begin{gather*}
2f_1(y)''-2c_1q(y)-(c_1y+c_2)q(y)'-2f_1(y)p(y)=0,\\
{2f_2(y)'+2c_1r(y)+}(c_1y+c_2)r(y)'+f_1(y)q(y)=0.
\end{gather*}
\item[$ii)$] a homothetic, non-Killing vector field if and only if $(c_1y+c_2)p'(y)+2c_1p(y)=0$ and~$X$ is described as in~\eqref{eqHomvector}, with $f_1$, $f_2$ explicitly determined as solutions of
\begin{gather*}
2f_1 ''(y)+(\frac{\eta}{2}-2c_1)q(y)-(c_1y+c_2)q '(y)-2f_1(y)p(y)=0,\\
{2f_2 '(y)+(2c_1-\eta)r(y)}+(c_1y+c_2)r'(y)+f_1(y)q(y)=0,\ \eta\neq0.
\end{gather*}
\item[$iii)$] a proper affine Killing vector field if and only if $(c_2y+c_3)p(y)'+2c_2p(y)=0$ and $X$ is described as in~\eqref{eqAffvector}, with $f_1$, $f_2$ explicitly determined as solutions of
\begin{gather*}
2f_1 ''(y){+\frac{c_1-3c_2}{2}q(y)}-(c_2y+c_3)q' (y)-2f_1(y)p(y)=0,\\
2f_2 '(y) {+(c_2-c_1)r(y)} +(c_2y+c_3)r '(y)+f_1(y)q(y)+c_4=0.
\end{gather*}
\item[$iv)$] a Ricci collineation if and only if $X_3=\frac{c_1}{\sqrt{|p(y)|}}$, where $c_1$ is a real constant.
\item[$v)$] a curvature collineation if and only if
\begin{gather*}
X_1=2f_1(y)t+\frac{c_1p'(y)}{2p(y)\sqrt{|p(y)|}}t-\frac12f_1 '(y)x^2-f_2 '(y)x+f_3(y),\\
X_2=f_1(y)x+f_2(y),\qquad
X_3=\frac{c_1}{\sqrt{|p(y)|}},
\end{gather*}
where $f_1(y)$ and $f_2(y)$ are arbitrary smooth functions on~$M$.
\end{itemize}
\end{Theorem}

\begin{proof}
Let $(M,g_f)$ be a conformally f\/lat strictly Walker manifold of dimension three, where~$g_f$ is described by the relation \eqref{metric}. As explained above, the function $f$ satisf\/ies $f(x,y)=p(y)x^2+q(y)x+r(y)$, ($p(y)\neq 0$), where $p(y)$, $q(y)$ and $r(y)$ are arbitrary smooth functions on~$M$. We then choose a~smooth vector f\/ield $X=X_1\partial_t+X_2\partial_x+X_3\partial_y$, where $X_1$, $X_2$ and $X_3$ are arbitrary smooth functions on~$M$. By Theorem~\ref{nablasymm}, $X$ is a homothetic vector f\/ield if and only if satisf\/ies equations~\eqref{eqHomvector} and~\eqref{eqHom}. Equation~\eqref{eqHom} for the above function $f(x,y)$ gives
\begin{gather*}
\left(2c_1p(y)+(c_1y+c_2)p(y)'\right)x^2 \\
\qquad {} -\left(2f_1(y)''+\left(\frac{\eta}{2}-2c_1\right)q(y)-(c_1y+c_2)q(y)'-2f_1(y)p(y)\right)x\\
\qquad{} +(2c_1-\eta)r(y)+2f_2(y)'+(c_1y+c_2)r(y)'+f_1(y)q(y)=0.
\end{gather*}
This equation immediately proves the second statement, since the coef\/f\/icients of $x$ and {its powers must vanish, in order to satisfy it identically. The f\/irst statement now follows by} {setting} $\eta=0$ in the equations of homothetic vector f\/ields.

With regard to af\/f\/ine Killing vector f\/ields, $X$ must satisfy equations \eqref{eqAffvector} and \eqref{eqAff}. So by straightforward calculations, the functions $p(y)$, $q(y)$ and $r(y)$ must satisfy
\begin{gather*}
\left(2c_2p(y)+(c_2y+c_3)p(y)'\right)x^2\\
\qquad{}-\left(2f_1 ''(y)-(c_2y+c_3)q'(y)-2f_1(y)p(y)+\frac{c_1-3c_2}{2}q(y)\right)x\\
\qquad{}+\left(2f_2 '(y)+(c_2y+c_3)r'(y)+f_1(y)q(y)+(c_2-c_1)r(y)+c_4\right)=0,
\end{gather*}
which {leads to} the third statement. Assertions (iv) and (v) are direct consequences of the cases~${\rm (b)}$ and~${\rm (b)'}$ of Theorem~\ref{RicciCol}, respectively.
\end{proof}

\section[Symmetries of homogeneous Lorentzian three-manifolds with recurrent curvature]{Symmetries of homogeneous Lorentzian three-manifolds \\ with recurrent curvature}\label{section4}

We reported the classif\/ication of homogeneous Lorentzian three-manifolds with recurrent curvature in Theorem~\ref{homex}. We shall now completely describe the symmetries of these manifolds. We start with the following.

\begin{Theorem}
Let $(M,g_f)$ be a homogeneous three-dimensional Lorentzian strictly Walker mani\-fold of type $\mathcal N_b$, that is, determined by
$f(x,y)=\frac{-2e^{bx}}{b^2}$, $b\neq 0$.

An arbitrary smooth vector field $X=X_1\partial_t+X_2\partial_x+X_3\partial_y$ on~$M$:
\begin{itemize}\itemsep=0pt
\item is homothetic $($equivalently, Killing$)$ if and only if
\begin{gather*}X_1=c_1t+c_2, \qquad X_2=\frac{2c_1}{b}, \qquad X_3=-c_1y+c_3.\end{gather*}
\item is affine if and only if
\begin{gather*}X_1=c_1t+c_2+c_4y,\qquad X_2=\frac{2c_1}{b},\qquad X_3=-c_1y+c_3.\end{gather*}
\item is a Ricci collineation if and only if $X_1$ is arbitrary and
\begin{gather*} X_2=-\frac2b f_1 '(y),\qquad X_3=f_1(y).\end{gather*}
\item is a curvature collineation if and only if{\samepage
\begin{gather*}X_1=f_2(y)-f_1 '(y) t+\frac2bf_1 ''(y)x,\qquad X_2=-\frac 2bf_1 '(y),\qquad X_3=f_1(y).\end{gather*}
In the above equations, $c_j $ are real constants and $f_1(y)$ an arbitrary smooth function.}
\end{itemize}
\end{Theorem}
\begin{proof}
Let $(M,g_f)$ be a homogeneous three-dimensional Lorentzian strictly Walker manifold of type $\mathcal N_b$, where $f(x,y)=\frac{-2e^{bx}}{b^2},\ b\neq 0$. By Theorem~\ref{nablasymm}, an arbitrary smooth vector f\/ield $X=X_1\partial_t+X_2\partial_x+X_3\partial_y$ is homothetic if and only if
$X_1$, $X_2$ and $X_3$ satisfy equations~\eqref{eqHomvector} and \eqref{eqHom}. Setting $f(x,y)=\frac{-2e^{bx}}{b^2}$ in equation~\eqref{eqHom}, we easily get
\begin{gather*}
\frac{e^{bx}}{b^2}(2-bx)\eta-\frac{2e^{bx}}{b^2}(2c_1+bf_1(y))-2f_1 ''(y)x=-2f_2 '(y),
\end{gather*}
which easily yields $\eta=0$, $ f_1(y)=-\frac{2{\rm c}_1}{b}$ and $f_2(y)=c_3$, where $c_3$ is a real constant. As $\eta=0$, $X$ is a homothetic vector f\/ield if and only if it is Killing, and the statement follows from equation~\eqref{eqHomvector}.

With regard to af\/f\/ine vector f\/ields, setting $f(x,y)=\frac{-2e^{bx}}{b^2}$ in equation~\eqref{eqAff} we f\/ind
\begin{gather*}
\frac{2e^{bx}}{b^2}\left(c_2-c_1+\frac{b}{2}(c_1+c_2)x+bf_1(y)\right)+2f_1 ''(y) x=2f_2 '(y)+c_4,
\end{gather*}
which since $b\neq 0$, {yields} the following relations
\begin{gather*}
c_2=-c_1,\qquad f_1(y)=\frac{2c_1}{b},\qquad f_2(y)=-\frac{c_4}{2}y+c_5,
\end{gather*}
where $c_5$ is a real constant. The statement then follows {if one chooses} suitable coef\/f\/i\-cients~$c_2$,~$c_3$ and $c_4$ in~\eqref{eqAffvector}.

Next, the result on Ricci collineations follows easily from the fact that they are characterized by equations
\begin{gather*}
\partial_tX_3=\partial_xX_3=0,\qquad \partial_yX_3+bX_2=0.
\end{gather*}
In particular, a Ricci collineation is also a curvature collineation when it satisf\/ies
\begin{gather*}
f_1 '(y)+\partial_tX_1=0,\qquad 2f_1 ''(y)-b\partial_xX_1=0,
\end{gather*}
which proves the last part of the statement.
\end{proof}

With regard to homogeneous three-dimensional Lorentzian strictly Walker manifolds of ty\-pe~$\mathcal P_c$ and~$\mathcal{CW}_\varepsilon$, comparing their def\/ining functions~$f(x,y)$ with the one of a locally conformally f\/lat strictly Walker three-manifold, it is easy to conclude that these homogeneous spaces are indeed locally conformally f\/lat. Therefore, their symmetries can be deduced as special cases of the results obtained in Theorem~\ref{cflatW}. In this way, we obtain the following.

\begin{Theorem}\label{Pc}
Let $(M,g_f)$ be a homogeneous three-dimensional Lorentzian strictly Walker mani\-fold of type~$\mathcal P_c$, that is, determined by $f(x,y)=-x^2\alpha(y)$, where \begin{gather*}\alpha '(y)=c\alpha(y)^{\frac32},\qquad \alpha(y)>0.\end{gather*}
Let $h(y)$ denote a smooth function explicitly determined from $\alpha(y)$ by equation
\begin{gather*}h''(y)+\alpha(y)h(y)=0.\end{gather*}

An arbitrary smooth vector field $X=X_1\partial_t+X_2\partial_x+X_3\partial_y$ on $M$:
\begin{itemize}\itemsep=0pt
\item is Killing if and only if
\begin{gather*}X_1=-h'(y) x+c_1,\qquad X_2=h(y),\qquad X_3=0.\end{gather*}
\item is homothetic if and only if
\begin{gather*}X_1=-h'(y)x+c_1+\eta t,\qquad X_2=h(y)+\frac\eta2 x,\qquad X_3=0.\end{gather*}
\item is affine if and only if
\begin{gather*}X_1=-h'(y)x+c_1+c_2t+c_3y,\qquad X_2=h(y)+\frac{c_2}{2}x,\qquad X_3=0.\end{gather*}
\item is a Ricci collineation if and only if $X_1$, $X_2$ are arbitrary and $X_3=\frac{c_1}{\sqrt{|\alpha(y)|}}$.
\item $X$ is a curvature collineation if and only if
\begin{gather*}
X_1=-\frac12f_1 '(y) x^2-f_2 '(y) x+\left(2f_1(y)+\frac{c_1c}{2}\right)t+f_3(y),\\ X_2=f_1(y)x+f_2(y),\qquad X_3=\frac{c_1}{\sqrt{|\alpha(y)|}}.
\end{gather*}
In the above equations, $f_i(y)$ are arbitrary smooth functions and $c_j$ are real constants.
\end{itemize}
\end{Theorem}

\begin{Theorem}\label{CW}
Let $(M,g_f)$ be a homogeneous three-dimensional Lorentzian strictly Walker mani\-fold of type $\mathcal{CW}_\varepsilon$, that is, determined by $f(x,y)=-\varepsilon x^2$. Consider the functions~$s(y)$,~$t(y)$ given by
\begin{gather*}s(y)=
\begin{cases}
c_2\sin(y)-c_1\cos(y) & \text{if} \ \varepsilon=1, \\ c_1e^{-y}-c_2e^{y} & \text{if} \ \varepsilon=-1 ,
\end{cases} \qquad
t(y)=
 \begin{cases}
c_1\sin(y)+c_2\cos(y) & \text{if} \ \varepsilon=1, \\ c_1e^{-y}+c_2e^{y} & \text{if} \ \varepsilon=-1.
\end{cases}
\end{gather*}
Then, an arbitrary smooth vector field $X=X_1\partial_t+X_2\partial_x+X_3\partial_y$ on $M$:
\begin{itemize}\itemsep=0pt
\item is Killing if and only if
\begin{gather*}
X_1=s(y)x+c_3,\qquad X_2=t(y),\qquad X_3=c_4,
\end{gather*}
\item is homothetic if and only if
\begin{gather*}
X_1=s(y)x+c_3+\eta t,\qquad X_2=t(y)+\frac{\eta}{2}x,\qquad X_3=c_4,
\end{gather*}
\item is affine if and only if
\begin{gather*}
X_1= s(y)x+c_3+c_5t+ c_6y,\qquad X_2=t(y)+\frac{c_5}{2}x,\qquad X_3=c_4,
\end{gather*}
\item is a Ricci collineation if and only if $X_3=c_1$.

\item is a curvature collineation if and only if
\begin{gather*}
X_1=2f_1(y)t-\frac12f_1 '(y) x^2-f_2 '(y) x+f_3(y),\qquad X_2=f_1(y)x+f_2(y),\qquad X_3=c_1.
\end{gather*}
In the above {equations, $f_i(y)$ are arbitrary smooth functions on $M$ and $c_j $} real constants.
\end{itemize}
\end{Theorem}

It is well known that for three-dimensional manifolds and locally conformally f\/lat manifolds, the curvature is completely determined by its Ricci curvature. However, as a consequence of Theo\-rem~\ref{cflatW} (in particular, Theorems~\ref{Pc} and~\ref{CW}), we f\/ind explicit examples of {three-dimensional} locally conformally f\/lat spaces, for which Ricci and curvature collineations are not equivalent.

\subsection*{Acknowledgements}

First author partially supported by funds of the University of Salento and MIUR (PRIN). Second author partially supported by funds of the University of Payame Noor. The authors wish to thank the anonymous referees for their valuable suggestions and comments.

\pdfbookmark[1]{References}{ref}
\LastPageEnding


\begin{thebibliography}{99}
\footnotesize\itemsep=-0.5pt

\bibitem{Heic}
Aichelburg P.C., Curvature collineations for gravitational {${\rm pp}$} waves,
 \href{http://dx.doi.org/10.1063/1.1665410}{\textit{J.~Math. Phys.}} \textbf{11} (1970), 2458--2462.

\bibitem{Ale}
Alekseevski D., Self-similar {L}orentzian manifolds, \href{http://dx.doi.org/10.1007/BF00054491}{\textit{Ann. Global Anal.
 Geom.}} \textbf{3} (1985), 59--84.

\bibitem{BCL}
Batat W., Calvaruso G., De~Leo B., Homogeneous {L}orentzian 3-manifolds with a
 parallel null vector f\/ield, \textit{Balkan~J. Geom. Appl.} \textbf{14}
 (2009), 11--20.

\bibitem{Beem}
Beem J.K., Proper homothetic maps and f\/ixed points, \href{http://dx.doi.org/10.1007/BF00419621}{\textit{Lett. Math. Phys.}}
 \textbf{2} (1978), 317--320.

\bibitem{bookW}
Brozos-V{\'a}zquez M., Garc{\'{\i}}a-R{\'{\i}}o E., Gilkey P.,
 Nik{\v{c}}evi{\'c} S., V{\'a}zquez-Lorenzo R., The geometry of {W}alker
 manifolds, \textit{Synthesis Lectures on Mathematics and Statistics}, Vol.~5,
 Morgan \& Claypool Publishers, Williston, VT, 2009.

\bibitem{CZ1}
Calvaruso G., Zaeim A., Invariant symmetries on non-reductive homogeneous
 pseudo-{R}iemannian four-manifolds, \href{http://dx.doi.org/10.1007/s13163-015-0168-8}{\textit{Rev. Mat. Complut.}} \textbf{28}
 (2015), 599--622.

\bibitem{CZ2}
Calvaruso G., Zaeim A., Geometric structures over four-dimensional generalized
 symmetric spaces, \href{http://dx.doi.org/10.1007/s13348-016-0173-3}{\textit{Collect. Math.}}, {t}o appear.

\bibitem{CSVV}
Calvi{\~n}o-Louzao E., Seoane-Bascoy J., V{\'a}zquez-Abal M.E.,
 V{\'a}zquez-Lorenzo R., Invariant {R}icci collineations on three-dimensional
 {L}ie groups, \href{http://dx.doi.org/10.1016/j.geomphys.2015.06.005}{\textit{J.~Geom. Phys.}} \textbf{96} (2015), 59--71.

\bibitem{CHK}
Camci U., Hussain I., Kucukakca Y., Curvature and {W}eyl collineations of
 {B}ianchi type {V} spacetimes, \href{http://dx.doi.org/10.1016/j.geomphys.2009.07.008}{\textit{J.~Geom. Phys.}} \textbf{59} (2009),
 1476--1484.

\bibitem{CS}
Camci U., Sharif M., Matter collineations of spacetime homogeneous
 {G}\"odel-type metrics, \href{http://dx.doi.org/10.1088/0264-9381/20/11/316}{\textit{Classical Quantum Gravity}} \textbf{20}
 (2003), 2169--2179, \href{http://arxiv.org/abs/gr-qc/0306129}{gr-qc/0306129}.

\bibitem{CCV}
Carot J., da~Costa J., Vaz E.G.L.R., Matter collineations: the inverse
 ``symmetry inheritance'' problem, \href{http://dx.doi.org/10.1063/1.530816}{\textit{J.~Math. Phys.}} \textbf{35} (1994),
 4832--4838.

\bibitem{ChGRVA}
Chaichi M., Garc{\'{\i}}a-R{\'{\i}}o E., V{\'a}zquez-Abal M.E.,
 Three-dimensional {L}orentz manifolds admitting a parallel null vector f\/ield,
 \href{http://dx.doi.org/10.1088/0305-4470/38/4/005}{\textit{J.~Phys.~A: Math. Gen.}} \textbf{38} (2005), 841--850.

\bibitem{FPP}
Flores J.L., Parra Y., Percoco U., On the general structure of {R}icci
 collineations for type {B} warped space-times, \href{http://dx.doi.org/10.1063/1.1775875}{\textit{J.~Math. Phys.}}
 \textbf{45} (2004), 3546--3557, \href{http://arxiv.org/abs/gr-qc/0405133}{gr-qc/0405133}.

\bibitem{GGN}
Garc{\'{\i}}a-R{\'{\i}}o E., Gilkey P.B., Nik{\v{c}}evi{\'c} S., Homogeneity of
 {L}orentzian three-manifolds with recurrent curvature, \href{http://dx.doi.org/10.1002/mana.201200302}{\textit{Math. Nachr.}}
 \textbf{287} (2014), 32--47, \href{http://arxiv.org/abs/1210.7764}{arXiv:1210.7764}.

\bibitem{Ha}
Hall G., Symmetries of the curvature, {W}eyl conformal and {W}eyl projective
 tensors on 4-dimensional {L}orentz manifolds, in Proceedings of the
 {I}nternational {C}onference ``{D}if\/ferential {G}eometry~-- {D}ynamical
 {S}ystems'' ({DGDS}-2007), \textit{BSG Proc.}, Vol.~15, Geom. Balkan Press,
 Bucharest, 2008, 89--98.

\bibitem{Hall}
Hall G.S., Symmetries and curvature structure in general relativity,
 \href{http://dx.doi.org/10.1142/1729}{\textit{World Scientific Lecture Notes in Physics}}, Vol.~46, World Scientif\/ic
 Publishing Co., Inc., River Edge, NJ, 2004.

\bibitem{HC}
Hall G.S., Capocci M.S., Classif\/ication and conformal symmetry in
 three-dimensional space-times, \href{http://dx.doi.org/10.1063/1.532815}{\textit{J.~Math. Phys.}} \textbf{40} (1999),
 1466--1478.

\bibitem{HLP}
Hall G.S., Low D.J., Pulham J.R., Af\/f\/ine collineations in general relativity
 and their f\/ixed point structure, \href{http://dx.doi.org/10.1063/1.530720}{\textit{J.~Math. Phys.}} \textbf{35} (1994),
 5930--5944.

\bibitem{HRV}
Hall G.S., Roy I., Vaz E.G.L.R., Ricci and matter collineations in space-time,
 \href{http://dx.doi.org/10.1007/BF02106969}{\textit{Gen. Relativity Gravitation}} \textbf{28} (1996), 299--310.

\bibitem{KR}
K{\"u}hnel W., Rademacher H.-B., Conformal {R}icci collineations of space-times,
 \href{http://dx.doi.org/10.1023/A:1013091621037}{\textit{Gen. Relativity Gravitation}} \textbf{33} (2001), 1905--1914.

\bibitem{L}
Levichev A.V., Methods for studying the causal structure of homogeneous
 {L}orentz manifolds, \href{http://dx.doi.org/10.1007/BF00970346}{\textit{Sib. Math.~J.}} \textbf{31} (1990), 395--408.

\bibitem{TA}
Tsamparlis M., Apostolopoulos P.S., Ricci and matter collineations of locally
 rotationally symmetric space-times, \href{http://dx.doi.org/10.1023/B:GERG.0000006693.75816.e9}{\textit{Gen. Relativity Gravitation}}
 \textbf{36} (2004), 47--69, \href{http://arxiv.org/abs/gr-qc/0309034}{gr-qc/0309034}.

\bibitem{W}
Walker A.G., On parallel f\/ields of partially null vector spaces,
 \href{http://dx.doi.org/10.1093/qmath/os-20.1.135}{\textit{Quart.~J. Math. Oxford Ser.}} \textbf{20} (1949), 135--145.

\end{thebibliography}
\end{document}